\documentclass[conference]{IEEEtran}
\ifCLASSINFOpdf
\else
\fi
\usepackage{mathrsfs}
\usepackage{amssymb}
\usepackage{amsmath,bm}

\usepackage{mathrsfs}
\usepackage{graphicx}
\usepackage{eurosym}
\usepackage{amssymb}
\usepackage{amsmath}
\usepackage{amsfonts}
\usepackage{epstopdf}
\usepackage{epsf,subfigure}
\usepackage{psfrag}
\usepackage{graphics}
\usepackage{color} 
\usepackage{cite}
\usepackage{array,multirow,pbox}
\usepackage{enumitem}
\usepackage{caption}

\newcommand{\rem}[1]{}
\newtheorem{proposition}{Proposition}
\newtheorem{definition}{Definition}
\newtheorem{theorem}{Theorem}

\newtheorem{lemma}{Lemma}

\newenvironment{proof}[1][Proof]{\begin{trivlist}
\item[\hskip \labelsep {\bfseries #1}]}{\end{trivlist}}

\newcommand{\qed}{\nobreak \ifvmode \relax \else
      \ifdim\lastskip<1.5em \hskip-\lastskip
      \hskip1.5em plus0em minus0.5em \fi \nobreak
      \vrule height0.75em width0.5em depth0.25em\fi}

\graphicspath{{Figures/}}

\begin{document}
\title{Decomposition of Nonlinear Dynamical Networks via Comparison Systems}

 \author{\IEEEauthorblockN{Abdullah Al Maruf}
 \IEEEauthorblockA{Washington State University\\
 Pullman, WA 99164 USA\\ Email:\,abdullahal.maruf@wsu.edu}
 \and
 \IEEEauthorblockN{Soumya Kundu, Enoch Yeung}
 \IEEEauthorblockA{Pacific Northwest National Laboratory\\
 Richland, WA 99354 USA\\
 Email:\,\{soumya.kundu,enoch.yeung\}@pnnl.gov,}
 \and
 \IEEEauthorblockN{Marian Anghel}
 \IEEEauthorblockA{Los Alamos National Laboratory\\
 Los Alamos, NM 87545 USA\\
 Email:\,manghel@lanl.gov}
 }


%\author{\IEEEauthorblockN{Abdullah Al Maruf\IEEEauthorrefmark{1},
%Soumya Kundu\IEEEauthorrefmark{2},
%Marian Anghel\IEEEauthorrefmark{3} and
%Enoch Yeung\IEEEauthorrefmark{4}}
%\IEEEauthorblockA{\IEEEauthorrefmark{1}School of Electrical Engineering and Computer Science, Washington State University,
%Pullman, WA 99164 USA\\ Email: abdullahal.maruf@wsu.edu}
%\IEEEauthorblockA{\IEEEauthorrefmark{2}Optimization and Control Group, Pacific Northwest National Laboratory, Richland, WA 99354 USA\\
%Email: soumya.kundu@pnnl.gov}
%\IEEEauthorblockA{\IEEEauthorrefmark{3}Information Sciences Group, Los Alamos National Laboratory, Los Alamos, NM 87545 USA\\
%Email: manghel@lanl.gov}
%\IEEEauthorblockA{\IEEEauthorrefmark{4}Data Sciences Group, Pacific Northwest National Laboratory, Richland, WA 99354 USA\\
%Email: enoch.yeung@pnnl.go}}


\maketitle

\begin{abstract}
In analysis and control of large-scale nonlinear dynamical systems, a distributed approach is often an attractive option due to its computational tractability and usually low communication requirements. Success of the distributed control design relies on the separability of the network into weakly interacting subsystems such that minimal information exchange between subsystems is sufficient to achieve satisfactory control performance. While distributed analysis and control design for dynamical network have been well studied, decomposition of nonlinear networks into weakly interacting subsystems has not received as much attention. In this article we propose a vector Lyapunov functions based approach to quantify the energy-flow in a dynamical network via a model of a comparison system. Introducing a notion of power and energy flow in a dynamical network, we use sum-of-squares programming tools to partition polynomial networks into weakly interacting subsystems. Examples are provided to illustrate the proposed method of decomposition.
\end{abstract}


% \fi
\IEEEpeerreviewmaketitle

%================================
\section{Introduction}
%================================

Control of large-scale dynamical systems is challenging because of the computational complexity. In such scenarios, distributed control design is an attractive option which often offers a trade off between computational tractability and controller performance \cite{Bullo:2009,Siljak:2010}. Distributed control design has become a critical challenge with the advent of interdependent large-scale infrastructure systems, smart cities, cloud architectures, and coordination and control problems for large fleets of independent agents. The combination of scale and dynamical complexity of these systems makes it computationally difficult to design global control solutions. Design of distributed control architectures can be roughly described as a three step process: 1) decomposition of the network into subsystems, 2) distributed control design to stabilize isolated subsystems, and 3) verification of stability of the closed-loop network under distributed control. The first step of this process is critical, and is the focus of this paper, since a poor system or model decomposition can affect all subsequent steps of distributed control design. 

In some applications, the physical layout or construction of a network dictates the subsystem structure. For example, in critical infrastructure systems the subsystem structure is traditionally defined based on distance and connectivity of buses, nodes or junctions. The increasing integration and instrumentation of heterogeneous infrastructure systems, e.g. power-gas systems \cite{chiang2016large,li2008interdependency} or water-power systems \cite{pereira2016joint}, introduces spatial and temporal overlap in time-scales across different systems, removing the insularity required previously to guarantee control performance. Moreover, there are scenarios where a suitable system decomposition may not be known {\it a priori}, e.g. design of large-scale or ad-hoc communication networks \cite{chiang2007layering, chiang2016large} or cyber-physical systems made of agile teams of agents. These reasons motivate the development of new system decomposition algorithms that identify appropriate subsystems that facilitate and enhance distributed control design and performance. 

While the idea of using a decomposed network model for distributed stability analysis and control design has been proposed in the literature \cite{Bullo:2009,Siljak:2010,Kundu:2015NecSys,Kundu:2017Multiple}\,, decomposition of nonlinear dynamical networks for distributed analysis and control has not received as much attention. In \cite{Anderson:2011,Anderson:2012}, authors proposed an algorithm to decompose nonlinear networks into sub-networks to perform stability analysis using composite Lyapunov functions. However, the proposed method involved linearization of the dynamics, and hence 1) is not suitable for networks that have nonlinear interactions, and 2) provides a decomposition that is only appropriate in the close neighborhood of the operating point. In a recently concluded work \cite{Liu:2017}, authors used data-driven deep learning techniques to decompose a nonlinear dynamical network using Koopman grammians. 

In this paper we propose a model-based decomposition method for nonlinear dynamical networks using comparison systems representation. Vector Lyapunov functions and comparison systems have been used for distributed stability analysis and control design for nonlinear networks \cite{Bailey:1966, Araki:1978, Michel:1983, Siljak:1972, Weissenberger:1973}. In \cite{Kundu:2015CDC,Kundu:2015NecSys,Kundu:2017Multiple} authors proposed sum-of-squares programming methods to compute the comparison systems for generic polynomial dynamical networks. In this paper, we use vector Lyapunov functions and comparison systems based approach to quantify the energy-flow in the network and perform a spectral clustering of the network such that the energy-flow between the sub-networks is minimized. The proposed method is dynamic since it allows to find a decomposition that best suits a given set of operating conditions. The paper is structured as follows. In Section\,\ref{S:problem}, we describe the problem. Section\,\ref{S:prelim} presents a brief overview of the necessary background. The proposed algorithms is discussed in details in Section\,\ref{S:method}, while numerical examples are presented in Section\,\ref{S:results}. The paper is concluded in Section\,\ref{S:concl}.

%==================================
\section{Problem Description}\label{S:problem}
%==================================
Consider an $n$-dimensional nonlinear dynamical system of the form 
\begin{align}\label{E:f}
\dot{x}(t) &= f(x(t))~ \forall t \geq 0\,,~x\in\mathbb{R}^n\,,
\end{align}
with an equilibrium at the origin $(f(0)= 0)$, where $f:\mathbb{R}^n \rightarrow \mathbb{R}^n$ is locally Lipschitz. For brevity, we would drop the argument $t$ from the state variables, whenever obvious. Often complex real-life systems can be modeled in the general form \eqref{E:f}. Analysis of such systems for control and decision making, however, can be non-trivial due to the nonlinearities. Especially, computational techniques for nonlinear systems analysis generally scale poorly with the size of the system. Furthermore, in real-time operations and decision making, often the information required for a global system-wide analysis is unavailable, rendering a decomposition-based distributed analysis the preferred option. The system in \eqref{E:f} can be viewed as a nonlinear network of $m$ dynamical subsystems as follows:
\begin{subequations}\label{E:fi}
\begin{align}
\mathcal{S}_i\,:\quad& \dot{x}_i=f_i(x_i)+\!\!\sum_{j\in\mathcal{N}_i\backslash{i}}\!\!g_{ij}(x_i,x_j)\,,\\
\text{where,}\,~&~x_i\in\mathbb{R}^{n_i},\,x_j\in\mathbb{R}^{n_j}~\,\forall j\in\mathcal{N}_i\backslash{i}\,,\\
&f_i(0)=0\text{ and }g_{ij}(x_i,0)=0~\,\forall x_i\,~\forall j\in\mathcal{N}_i\backslash{i}\,.
\end{align}\end{subequations}
Each subsystem $\mathcal{S}_i$ has some local dynamical state variables $x_i$ associated with it, while $\mathcal{N}_i$ (with $i\in\mathcal{N}_i$) represent the neighboring nodes that interact with the subsystem-$i$. $f_i:\mathbb{R}^{n_i}\rightarrow\mathbb{R}^{n_i}$ represents the isolated subsystem dynamics, while the terms $g_{ij}:\mathbb{R}^{n_i}\!\times\!\mathbb{R}^{n_j}\!\rightarrow\!\mathbb{R}^{n_i}$ represent the interactions from the neighbors. Note that the interaction model (total interaction as a sum of the $g_{ij}$'s) used here is chosen for simplicity, but more generic interaction models can also be considered. Also, it is assumed that the decomposition is overlapping in a way that no two subsystems share any common state variable, and 
\begin{align*}
{\sum}_{i=1}^mn_i=n\,. 
\end{align*}

A decomposed system model of the form \eqref{E:fi} has been used for stability analysis studies in the literature (see \cite{Araki:1978,Michel:1983,Siljak:1972,Weissenberger:1973,Kundu:2017Multiple,Kundu:2015ACC}). However, finding a suitable decomposition of the large system that facilitates distributed analysis is not trivial. In \cite{Anderson:2012,Anderson:2011} authors applied spectral graph partitioning algorithms on linearized system dynamics to obtain a system decomposition that aims to minimize worst-case energy flows between subsystems. However, because of the linearization, such methods are only applicable in the close neighborhood of the operating points. For systems under large disturbance, i.e. far away from the operating point, the decomposition structure is expected to change. Furthermore, quadratic and higher order terms in the interactions do not show up in the linearized representation and are therefore completely ignored in the decomposition.

In this article, we propose a decomposition algorithm to partition a nonlinear dynamical network into weakly-interacting subsystems, whereby the nonlinearity of the system is explicitly considered in the decomposition algorithm, via the use of a \textit{comparison systems} representation, as explained in the following sections. 

%================================
\section{Preliminaries}\label{S:prelim}

Let us use $\left|\,\cdot\,\right|$ to denote both the Euclidean norm of a vector and the absolute value of a scalar; and use $\mathbb{R}\left[x\right]$ to denote the ring of all polynomials in $x\in\mathbb{R}^n$. 

%---------------------------------------------------------------------
\subsection{Stability Analysis}

The equilibrium point at the origin of \eqref{E:f} is Lyapunov stable if, for every $\varepsilon\!>\!0$ there is a $\delta\!>\!0$ such that $\left\vert x(t)\right\vert\!<\!\varepsilon~\forall t\!\geq\! 0$ whenever $\left\vert x(0)\right\vert\!<\!\delta\,.$ Moreover, it is asymptotically stable in a domain $\mathcal{D}\!\subseteq\!\mathbb{R}^n,\,0\!\in\!\mathcal{D},$ if it is Lyapunov stable and  $\lim_{t\rightarrow\infty}\left\vert x(t)\right\vert \!=\!0\,$ for every $x(0)\!\in\!\mathcal{D}$\,.
%and it is exponentially stable if there exists $b,\,c > 0$ such that $\left\vert x(t)\right\vert <ce^{-bt}\left\vert  x(0)\right\vert \,\,\forall t\geq 0\,$, for every $\left\vert x(0)\right\vert \in \mathcal{D}$.

\begin{theorem}\label{T:Lyap}
(Lyapunov, \cite{Lyapunov:1892}, \cite{Khalil:1996}) If there exists a domain $\mathcal{D}\!\!\subseteq\!\!\mathbb{R}^n$, $0\!\in\!\!\mathcal{D}$, and a continuously differentiable positive definite function {$\tilde{V}\!\!:\!\mathcal{D}\!\rightarrow\! \mathbb{R}_{\geq 0}$}, i.e. the `Lyapunov function' (LF), then the origin of \eqref{E:f} is asymptotically stable if $\nabla{\tilde{V}}^T\!\!f(x)$ is negative definite in $\mathcal{D}$. 
%and is exponentially stable if $\nabla{\tilde{V}}^T\!\!f(x)\leq\!-\alpha\, \tilde{V}~\forall x\!\in\!\mathcal{D}$, for some $\alpha>0$.
\end{theorem}

An estimate of the region-of-attraction (ROA) can be given by \cite{Genesio:1985}
\begin{align}\label{E:ROA} 
&~~\mathcal{R}:=\left\lbrace x\in\mathcal{D}\left| {V}(x)\leq 1\right.\right\rbrace\,,~\text{with}~{V}(x)= \frac{\tilde{V}(x)}{\gamma^{max}},\\
&\text{where}~\gamma^{max}:=\max\left\lbrace \gamma\, \left\vert\, \left\lbrace x\in\mathbb{R}^n\left| \tilde{V}(x)\leq\gamma\right.\right\rbrace \subseteq \mathcal{D}\right.\right\rbrace\,,\notag
\end{align}
i.e. the boundary of the ROA is estimated by the unit level-set of a suitably scaled LF ${V}(x)$. 

Relatively recent studies have explored how sum-of-squares (SOS) based methods can be utilized to find LFs by restricting the search space to SOS polynomials \cite{Wloszek:2003,Parrilo:2000,Tan:2006,Anghel:2013}. A (multivariate) polynomial $p\in\mathbb{R}\left[x\right],~x\in\mathbb{R}^n$, is called a \textit{sum-of-squares} (SOS) if there exist some polynomial functions $h_i(x), i = 1\ldots s$ such that 
$p(x) = \sum_{i=1}^s h_i^2(x)$.
We denote the ring of all SOS polynomials in $x\in\mathbb{R}^n$ by ${\Sigma}[x]$. 
Checking if $p\!\in\!\mathbb{R}[x]$ is an SOS is a semi-definite problem which can be solved with a MATLAB$^\text{\textregistered}$ toolbox SOSTOOLS \cite{sostools13,Antonis:2005a} along with a semidefinite programming solver such as SeDuMi \cite{Sturm:1999}. The SOS technique can be used to search for polynomial LFs by translating the conditions in Theorem\,\ref{T:Lyap} to equivalent SOS conditions \cite{sostools13,Wloszek:2003,Wloszek:2005,Antonis:2005,Antonis:2005a, Chesi:2010a }. An important result from algebraic geometry, called Putinar's Positivstellensatz (P-Satz) theorem \cite{Putinar:1993,Lasserre:2009}, helps in translating the SOS conditions into SOS feasibility problems.
\begin{theorem}\label{T:Putinar}
Let $\mathcal{K}\!\!=\! \left\lbrace x\in\mathbb{R}^n\left\vert\, k_1(x) \geq 0\,, \dots , k_m(x)\geq 0\!\right.\right\rbrace$ be a compact set, where $k_j\!\in\!\mathbb{R}[x]$, $\forall j\in\left\lbrace 1,\dots,m\right\rbrace$. Suppose there exists a $\mu\!\in\! \left\lbrace \sigma_0 + {\sum}_{j=1}^m\sigma_j\,k_j \left\vert\, \sigma_0,\sigma_j \in \Sigma[x]\,,\forall j \right. \right\rbrace$ such that $\left\lbrace \left. x\in\mathbb{R}^n \right\vert\, \mu(x)\geq 0 \right\rbrace$ is compact. Then,  
\begin{align*}
p(x)\!>\!0~\forall x\!\in\!\mathcal{K}
\!\implies\! p \!\in\! \left\lbrace \sigma_0 \!+\! {\sum}_j\sigma_jk_j\!\left\vert\, \sigma_0,\sigma_j\!\in\!\Sigma[x],\forall j\!\right.\right\rbrace\!.
\end{align*}
\end{theorem}

\subsection{Linear Comparison Principle}
{In} \cite{Conti:1956,Brauer:1961} the authors proposed to view the LF as a dependent variable in a first-order auxiliary differential equation, often termed as the `comparison equation' (or, `comparison system'). 
%Soon after, it
{It was shown in \cite{Bellman:1962,Bailey:1966} that, under certain conditions,} the comparison equation can be effectively reduced to a set of linear differential equations. 
Noting that all the elements of the { matrix} $e^{At},~ t\geq 0$, where $A=\left[a_{ij}\right]\in\mathbb{R}^{m\times m}$, are non-negative if and only if $a_{ij}\geq 0, i\neq j$, {it was shown in \cite{Beckenbach:1961,Bellman:1962}:
%the authors in \cite{Beckenbach:1961,Bellman:1962} proposed the following:
\begin{lemma}\label{L:comparison}
Let $A\!\in\!\mathbb{R}^{m\times m}$ have non-negative off-diagonal elements, $v:[0,\infty)\!\rightarrow\!\mathbb{R}^m$ and $r:[0,\infty)\!\rightarrow\!\mathbb{R}^m$. If $v(0)\!=\!r(0)\,$, $\dot{v}(t)\!\leq\!A\,v(t)$ and $\dot{r}(t)\!=\!A\,r(t)\,,$ then $v(t)\!\leq\! r(t)~\forall t\!\geq\! 0\,$.
\end{lemma}}
We henceforth refer to Lemma\,\ref{L:comparison} as the `comparison principle'; the {linear time-invariant system $\dot{r}(t)\!=\!A\,r(t)$}
%relation \eqref{E:comp_ineq} 
as a `comparison system' (CS); and the matrix $A$ as the comparison matrix (CM). 

\subsection{Similarity Graphs and Spectral Partitioning}

Consider a graph $\mathcal{G}(\mathcal{V},\mathcal{E})$ where $\mathcal{V}= \{1, 2,\dots, m\}$ represent a set of $m$ vertices (or, nodes) and $\mathcal{E}\subseteq\mathcal{V}\times\mathcal{V}$ denote a set of $l$ edges between the vertices (i.e. $\mathbf{card}({\mathcal{E}})=l$). A notion of \textit{similarity} (or, edge weight), denoted by a scalar $w_{ij}\geq 0$\,, is associated with each pair of nodes $(i,j)$ in the graph, such that $w_{ij}>0$ if an edge exists between the nodes $i$ and $j$, i.e. $(i,j)\in\mathcal{E}$, and $w_{ij}=0$ if there does not exist an edge between the nodes $i$ and $j$ (i.e. $(i,j)\notin\mathcal{E}$). The edge weights quantify how similar two nodes in the graph are to each other, i.e. higher the value of the edge weights more similar the corresponding vertices are. 
%The \textit{graph incidence matrix} $C(\mathcal{G})=[c_{dk}]\in\mathbb{R}^m\times\mathbb{R}^l$ is constructed such that the entries in the $k$-th column of the incidence matrix corresponding to the $k$-th edge $E_k=(i,j)$ between nodes $i$ and $j$ are given by
%\begin{align*}
%c_{dk}=\left\lbrace\begin{array}{ll}+1,&d=i\\
%-1,&d=j\\
%0,&\text{otherwise}\end{array}\right.~\forall k\text{ such that }E_k=(i,j)\in\mathcal{E}\,.
%\end{align*}

The spectral clustering technique used in this paper for graph partitioning is applicable for graphs that are \textit{undirected}, i.e. the edge weights (or similarity values) satisfy $w_{ij}=w_{ji}\,\forall i\neq j$ (clearly, for undirected graphs, $(i,j)\in\mathcal{E}$ implies that $(j,i)\in\mathcal{E}$)\,. If a graph is \textit{directed}, we can convert it into an \textit{undirected} graph by assigning
\begin{align*}
w_{ij}\gets\frac{1}{2}\left(w_{ij}+w_{ji}\right)\,,\text{ or }w_{ij}\gets\max\lbrace w_{ij},w_{ji}\rbrace\,.
\end{align*}
A symmetric \textit{weighted adjacency matrix} $W(\mathcal{G})=[w_{ij}]\in\mathbb{R}^{m\times m}$ is then constructed, such that $w_{ii}=0~\forall i\in\mathcal{V}$\,, and $w_{ij}=w_{ji}>0~\forall (i,j)\in\mathcal{E}\,.$ Finally, a normalized (symmetric positive semi-definite) \textit{graph Laplacian matrix}, $L_{sym}(\mathcal{G})\in\mathbb{R}^{m\times m}$, is defined as,
\begin{align*}
L_{sym} = I_{m}-D^{-1/2}WD^{-1/2},\,~\,D=\mathbf{diag}\left(W\,\mathbf{1}_m\right)\,,
\end{align*}
where $I_m$ is an $m\times m$ identity matrix, $\mathbf{1}_m$ is a $m$-dimensional column vector with each entry equal to $1$\,, $D$ is a diagonal matrix with diagonal entries equal to the sum of the corresponding row in $W$\,; and $\left(\cdot\right)^{-1/2}$ denotes the inverse of the square root a matrix.

The idea behind spectral graph partitioning is to partition the graph $\mathcal{G}$ into sub-graphs in such a way that the vertices within a sub-graph are \textit{similar} to each other (i.e. have high edge weights), while vertices from two different sub-graphs are \textit{dissimilar} (i.e. have low edge weights). In particular, the $K$-means spectral clustering algorithm uses the normalized graph Laplacian $L_{sym}$ to partition the vertices into $K$ clusters. The first $K$ eigenvectors of the graph Laplacian, in the order of increasing eigenvalues, are used as columns to construct a matrix $U\in\mathbb{R}^{m\times K}$. Then the normalized rows of the matrix $U$ are clustered using the $K$-means clustering algorithm. For further details on the algorithm, please refer to \cite{Von:2007,Ng:2002}.

%=======================================
\section{Decomposition via Comparison Systems}\label{S:method}
%=======================================

In this section we describe our proposed approach of decomposing a nonlinear (polynomial) dynamical network using a comparison system. The proposed approach is composed of two steps - 1) computing a comparison system, and 2) apply spectral graph partitioning on the comparison system.

%-----------------------------------------------------------------------
\subsection{Constructing Comparison System} 
%-----------------------------------------------------------------------
Given a polynomial network of the form \eqref{E:fi}\,, vector Lyapunov functions of the subsystems can be used to construct a linear comparison system. In \cite{Kundu:2015CDC} authors used sum-of-squares programming to compute the vector Lyapunov functions and the comparison matrix for polynomial dynamical networks. It is assumed that, in absence of any interaction, each subsystem $\mathcal{S}_i$ is locally asymptotically stable at the origin. Specifically, we assume that for each $i\in\lbrace 1,\dots,m\rbrace$
\begin{align*}
\mathcal{S}_i\,\text{ (isolated) }\,:~\dot{x}_i=f_i(x_i)
\end{align*} 
has an asymptotically stable equilibrium point at the origin, and admits a polynomial Lyapunov function $V_i(x_i)$\,. The expanding interior algorithm \cite{Wloszek:2003,Anghel:2013} can be used to compute these LFs and a corresponding region of attraction such that $\mathcal{R}_i:= \left\lbrace x_i\in\mathbb{R}^{n_i}\left| V_i(x_i)\leq 1\right.\right\rbrace$ is an estimate of the region of attraction of the $i$-th isolated subsystem.

The vector LFs thus computed are used to define the states of the \textit{comparison system}, 
\begin{align*}
v(t)&:=\begin{bmatrix}
V_1(x_1(t)) & V_2(x_2(t)) & \dots & V_m(x_m(t))
\end{bmatrix}^T
\end{align*}
such that the comparison system is a \textit{positive system}, i.e. each state of the comparison system only takes non-negative values at all time $t$\,. Then the objective is to compute the comparison matrix (CM) $A=[a_{ij}]$ such that in some domain 
\begin{align*}
\mathcal{D}&:= \left\lbrace x\in\mathbb{R}^n\left| \,V_i(x_i)\leq \gamma_i~\forall i\right.\right\rbrace\subseteq \left\lbrace x\in\mathbb{R}^n\left| \,x_i\in\mathcal{R}_i~\forall i\right.\right\rbrace
\end{align*}
where $\gamma_i\leq 1$\,, the following comparison equations hold:
\begin{align*}
\dot{v}\leq A\,v\,~\forall x\in\mathcal{D}\,,\text{ where }A=[a_{ij}]\,,\,a_{ij}\geq 0\,\forall i\neq j\,.
\end{align*}

Sum-of-squares programming can be used to compute the CM which has non-negative off-diagonal elements. In order to obtain a stable comparison system, the CM $A$ has to be Hurwitz. While solving for a generic Hurwitz CM for a large system is difficult, a sufficient condition for CM to be Hurwitz is given by the Gershgorin circle theorem \cite{Bell:1965} which says that a strictly diagonally dominant matrix\footnote{$A=[a_{ij}]$ is strictly diagonally dominant if $\sum_{j\neq i}\left|a_{ij}\right|<\left|a_{ii}\right|,\forall i$.} with negative diagonal elements is Hurwitz. This condition motivates us to solve for each row of the CM, in a parallel and scalable way, by solving the following problem:
\begin{align}
\forall i: ~\underset{a_{ij}}{\min}~{\sum}_{j\in\mathcal{N}_i} a_{ij},\text{ s.t. }\,\dot{V}_i\leq {\sum}_{j\in\mathcal{N}_i} a_{ij}V_j~\forall x\in\mathcal{D},
\end{align}
which can be formulated as an SOS optimization:
\begin{align}
\forall i\!:& ~\underset{a_{ij},\sigma_{ij}}{\text{minimize}}~{\sum}_{j\in\mathcal{N}_i} a_{ij}\\
\text{s.t.}&~ a_{ii}\in\mathbb{R}\,,~a_{ij}\!\geq 0~\forall i\neq j\,,~\sigma_{ij}\in\Sigma[x_{\mathcal{N}_i}]\,,\notag\\
&\!-\!\nabla V_i^T\!(f_i\!+\!\sum_{j\neq i} g_{ij}\!)\!+\!\!\sum_{j\in\mathcal{N}_i}\!\!\left(a_{ij}V_j\!-\!\sigma_{ij}(\gamma_j\!-\!V_j)\right)\!\in\!\Sigma[x_{\mathcal{N}_i}].\notag
\end{align}
Here $\nabla$ is the gradient operator, and $x_{\mathcal{N}_i}$ denote the state variables that belong to the neighborhood of the subsystem $i$ (recall that $i\in\mathcal{N}_i$\,). The CM is guaranteed to be Hurwitz, if the sum of each row of the CM $A$ turns out to be negative. Otherwise, one can do an eigenvalue analysis to determine whether or not the CM is Hurwitz. 

Note that the CM is not unique, and depends on the vector LFs computed earlier, as well as the domain of definition (also referred to as the domain of interest) $\mathcal{D}$\,. For a given set of vector LFs, the CM could be different if the domain of interest changes. This particular feature is useful in monitoring the change in the energy-flow pattern in a nonlinear dynamical network under different operating conditions.

%---------------------------------------------------------------------------
\subsection{Energy-Flow in the Network} 
%---------------------------------------------------------------------------

The comparison system (CS)
\begin{align}\label{E:CS}
\dot{r}(t)&=A\,r(t)\,,\quad r(0)=v(0)\,,
\end{align}
where we denote the CS states as $r=\left[
r_1\,,\,\dots\,\, r_m
\right]^T\!\!,$ provides an upper bound on the values of the level-set of the subsystem vector LFs, i.e. $V_i(x_i(t))\leq r_i(t)~\forall i\,\forall t\,.$ The comparison system can be represented as a \textit{similarity graph} $\mathcal{G}(\mathcal{V},\mathcal{E})$, where $\mathcal{V}=\lbrace i\rbrace_{i=1}^m$ represent the nodes, while $\mathcal{E}=\lbrace E_k\rbrace_{k=1}^l$ denote the directional edges (where $l$ is the number of positive off-diagonal elements in the CM). We consider the case when the CM $A$ is Hurwitz and the original system \eqref{E:fi} is guaranteed to be asymptotically stable. The goal here is to partition the network such that the \textit{worst case} energy flow between the sub-networks is minimized. However, for generic nonlinear dynamical networks of the form \eqref{E:fi}, the concept of \textit{energy-flow} along the edges is not well defined, but are usually informed by physical knowledge of the system (e.g. in power systems, the edge flow could be represented by the active power flowing from one node to another). It may be convenient to view the vector LFs as some form of energy-levels of the subsystems of the network, while the CS provides some bound on the rate of change of the energy levels of the subsystems. We use the vector LFs and their time derivatives to quantify the energy flow across edges as follows: 

\begin{definition}
Let us define the \textit{`power-flow'}, $\phi_{ij}(t)$\,, from node $j$ into the node $i$ along a directional edge $(i,j)$ in the dynamical network \eqref{E:fi} which admits vector LFs $V_i(x_i)\,\forall i$\,, as \begin{align*}
\phi_{ij}(t):=\nabla V_i^T g_{ij}(x_i,x_j)\,,
\end{align*}
and \textit{energy-flow}, $\psi_{ij}(t_0,t_f)$\,, across the directional edge $(i,j)$ between time $t_0$ and $t_f$ is defined as the following integral,
\begin{align*}
\psi_{ij}(t_0,t_f):=\int_{t_0}^{t_f}\left|\,\phi_{ij}(t)\,\right|\,dt\,.
\end{align*}
\end{definition}

Note that the \textit{power-flow} can have both positive and negative values, where the positive values are referred to as the \textit{`in-flow'} of power, while negative values are referred to as \textit{`out-flow'} of power. The \textit{energy-flow} is the time integral of the absolute value of the power-flow. Computing the power and energy-flows requires solving the nonlinear dynamics for the given initial conditions. However, for a given set of possibly (or likely) initial conditions, we can compute upper bounds on the power and energy-flow using the vector LFs. Specifically, note that, since the dynamical equations are polynomial, and if the states are bounded, we can find finite positive scalars $\alpha_{ij}\in\mathbb{R}_{\geq 0}\,\forall (i,j)\in\mathcal{E}$\,, such that 
\begin{align}
x\in\mathcal{D}\implies\left|\phi_{ij}(t)\right|\leq \alpha_{ij}\,V_j(x_j)\,~\,\forall (i,j)\in\mathcal{E}\,.
\end{align}
Recall that $x_j\!=\!0$ implies $V_j(0)\!=\!0$ and $\phi_{ij}\!=\!0$\,. The scalars $\alpha_{ij}$ are bounded (since $\nabla V_i^T$ and $g_{ij}(\cdot)$ are polynomials), but not unique and depend on the domain of interest. Since the evolution of the vector LFs are governed by the CS, a closed-form expression can be found for the upper bound of the energy-flow in the network within a domain of interest. Specifically, let us define the observable (or, output) vector of the CS as
\begin{align}
y&=C^Tr\,,\quad C=\begin{bmatrix}
c_1 & c_2 &\dots &c_l
\end{bmatrix}\in\mathbb{R}_{\geq 0}^{m\times l}
\end{align}
where each column of $C$ corresponds to the a different edge in the network. If the $k$-th edge $(i,j)\in\mathcal{E}$ is from the node $j$ to the node $i$\,, then the $k$-th column is given by $c_k=\alpha_{ij}e_j$ where $e_j$ denotes an $m$-dimensional \textit{standard basis vector} with the $j$-th entry equal to $1$ and all other entries equal to $0$\,. 

\begin{proposition}
Suppose that a nonlinear dynamical network \eqref{E:fi} is represented by the CS \eqref{E:CS} in some domain of interest $\mathcal{D}$ with a Hurwitz CM. Given a trajectory within $\mathcal{D}\,$, the net energy-flow in the network over time $t\in[0,\infty)$ is upper bounded by $-(C\,\mathbf{1}_l)^TA^{-1}v(0)$\,, while the energy-flow in the directional edge $(i,j)$ is upper bounded by $-\alpha_{ij}e_j^TA^{-1}v(0)$\,.
\end{proposition}
\begin{proof}
The CS is a \textit{positive system} with a CM that is Metzler (i.e. non-negative off-diagonal elements). If the CM is also Hurwitz, its inverse exists and has only non-positive entries \cite{Weissenberger:1973}. Therefore, the integral $\int_0^{\infty}\exp(At)dt$ is equal to $-A^{-1}$ with all non-negative entries. Note that,
\begin{align*}
\psi_{ij}(0,\infty)&=\int_0^\infty\left|\phi_{ij}(t)\right|\,dt\leq \int_0^\infty \alpha_{ij}V_j(x_j(t))\,dt\\
&\leq \int_0^\infty\alpha_{ij}r_j(t)\,dt=\int_0^\infty \alpha_{ij}e_j^Tr(t)\,dt\\
&=\alpha_{ij}e_j^T\int_0^\infty e^{At}\,dt\,r(0)=-\alpha_{ij}e_j^TA^{-1}v(0)
\end{align*}
The total energy-flow across all edges is given by 
\begin{align*}
\sum_{(i,j)\in\mathcal{E}}\!\!\!\!\psi_{ij}(0,\infty)\leq\sum_{k=1}^l-c_k^TA^{-1}v(0)=-(C\,\mathbf{1}_l)^TA^{-1}v(0)\,.\!\!\!\!\!\qed
\end{align*}\end{proof}

Note that, since $-A^{-1}$ has only non-negative entries, the upper bound on the energy flow in any edge is monotonically increasing with respect to the initial level-sets, $v(0)$, of the vector LFs. This is useful if we want to compute the \textit{worst-case} energy flow in the network given some domain of interest. For example, if the domain of interest is defined in the form of $\mathcal{D}= \left\lbrace x\in\mathbb{R}^n\left| \,V_i(x_i)\leq \gamma_i~\forall i\right.\right\rbrace$, for some $\gamma_i\leq 1$\,, then the \textit{worst-case} upper bound on the energy flow across the $k$-th edge is $-c_k^TA^{-1}\Gamma$\,, where $\Gamma=\begin{bmatrix}
\gamma_1&\dots&\gamma_m
\end{bmatrix}^T$.

%---------------------------------------------------------------------------
\subsection{$K$-means Clustering on Comparison System}
%---------------------------------------------------------------------------

The formulation of energy-flow along the edges proposed here allows us to quantify the energy-flow pattern in the network given an initial condition. It is expected that, under different operating conditions, the energy-flow pattern of the network changes and henceforth the decomposition that minimizes the energy-flow between sub-networks may be different as well. While the exact operating conditions may not be known \textit{apriori}, one can proceed with the analysis of the \textit{worst-case} energy-flow within some domain of interest. 

For a domain of interest $\mathcal{D}\!=\! \left\lbrace x\in\mathbb{R}^n\left| \,v(t)\leq \Gamma~\forall t\right.\right\rbrace$\,, where $\Gamma\leq 1$ (element-wise), we use the \textit{worst-case} upper bound on the energy flow across the edge in either direction to construct a symmetric \textit{weighted adjacency matrix} of the network. Specifically, we construct the symmetric weighted adjacency matrix $W(\mathcal{G})=[w_{ij}]$ as,
\begin{align*}
w_{ij}=\left\lbrace\begin{array}{c} \max\lbrace -\alpha_{ij}e_j^TA^{-1}\Gamma\,,\,-\alpha_{ji}e_i^TA^{-1}\Gamma\rbrace,\quad (i,j)\in\mathcal{E}\\
0\,,\qquad\qquad \text{otherwise}
\!\!\end{array}\!\!\right.\!\!
\end{align*}

Note, of course, that while this construction of weighted adjacency matrix is concerned with the \textit{worst-case} energy-flow in a domain of definition, the formulation can be adapted to investigate the energy-flow under a given set of initial conditions, by replacing $\Gamma$ with $v(0)$\,. The final step involves choosing the desired number of sub-networks and applying the $K$-means clustering algorithm on the weighted adjacency matrix $W$\,.

%=======================================
\section{Simulation Results}\label{S:results}
%=======================================
In this section we present some numerical examples to illustrate how the proposed method can be used to compute the energy-flow across nonlinear dynamical networks under varying operating conditions, and decompose the network minimizing energy-flow between sub-networks. In the figures presented below, the width of the line between any two nodes represents the relative values of the energy flowing through the edge (i.e. thicker lines represent higher energy-flow).

%---------------------------------------------------------------------------
\subsection{Lotka-Volterra System}
%---------------------------------------------------------------------------

Our first example is a Lotka-Volterra system that describes the evolution of population in a network of $16$ communities:
\begin{align}\label{E:LV}
\dot{x}_i&=\left(b_i-x_i\right)x_i-{\sum}_{j\in\mathcal{N}_i\backslash\lbrace i\rbrace}x_j\left(c_{ij}+d_{ij}x_i\right)
\end{align}
where $i=1,2,\dots,16,$\,. We use a similar model as in \cite{Anderson:2012}, with slightly modified network structure and interaction coefficients ($c_{ij}$ and $d_{ij}$) that are scaled down by a factor of $0.3$\,, in order to facilitate the construction of comparison systems over a wider range of operating conditions. After shifting the equilibrium point of the system to origin, the Lotka-Volterra system is expressed in the form \eqref{E:fi}. Quadratic vector LFs $V_i(x_i)$ are computed for each community, such that $\mathcal{R}_i=\left\lbrace x_i\left|\,V_i(x_i)\leq 1\right.\right\rbrace$ is an estimation of the region of attraction. Defining the domain of interest as $\mathcal{D}=\left\lbrace x\in\mathbb{R}^{16}\left|\,V_i(x_i)\leq \gamma_i~\forall i\right.\right\rbrace$, for some $\gamma_i\leq 1$\,, a CS is constructed based on which the energy-flow pattern of the network can be computed. 

\begin{figure*}[thpb]
\centering
\captionsetup{justification=centering}
\subfigure[$\gamma_i=0.01$ (calculated)]{
\includegraphics[scale=0.25]{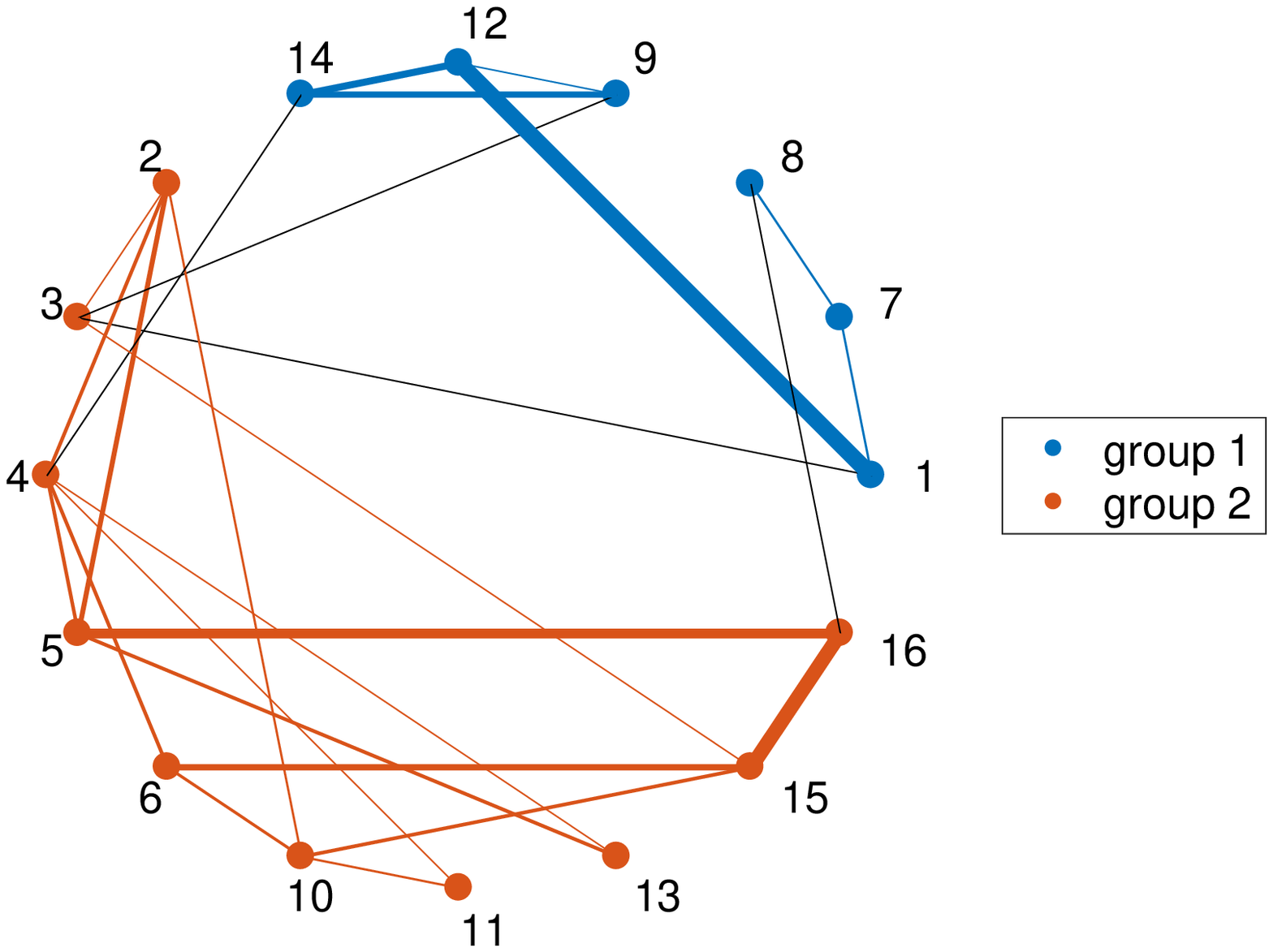}\label{F:LV_low_cal}
}
\hspace{0.01in}
\subfigure[$\gamma_i=0.01$ (simulated)]{
\includegraphics[scale=0.25]{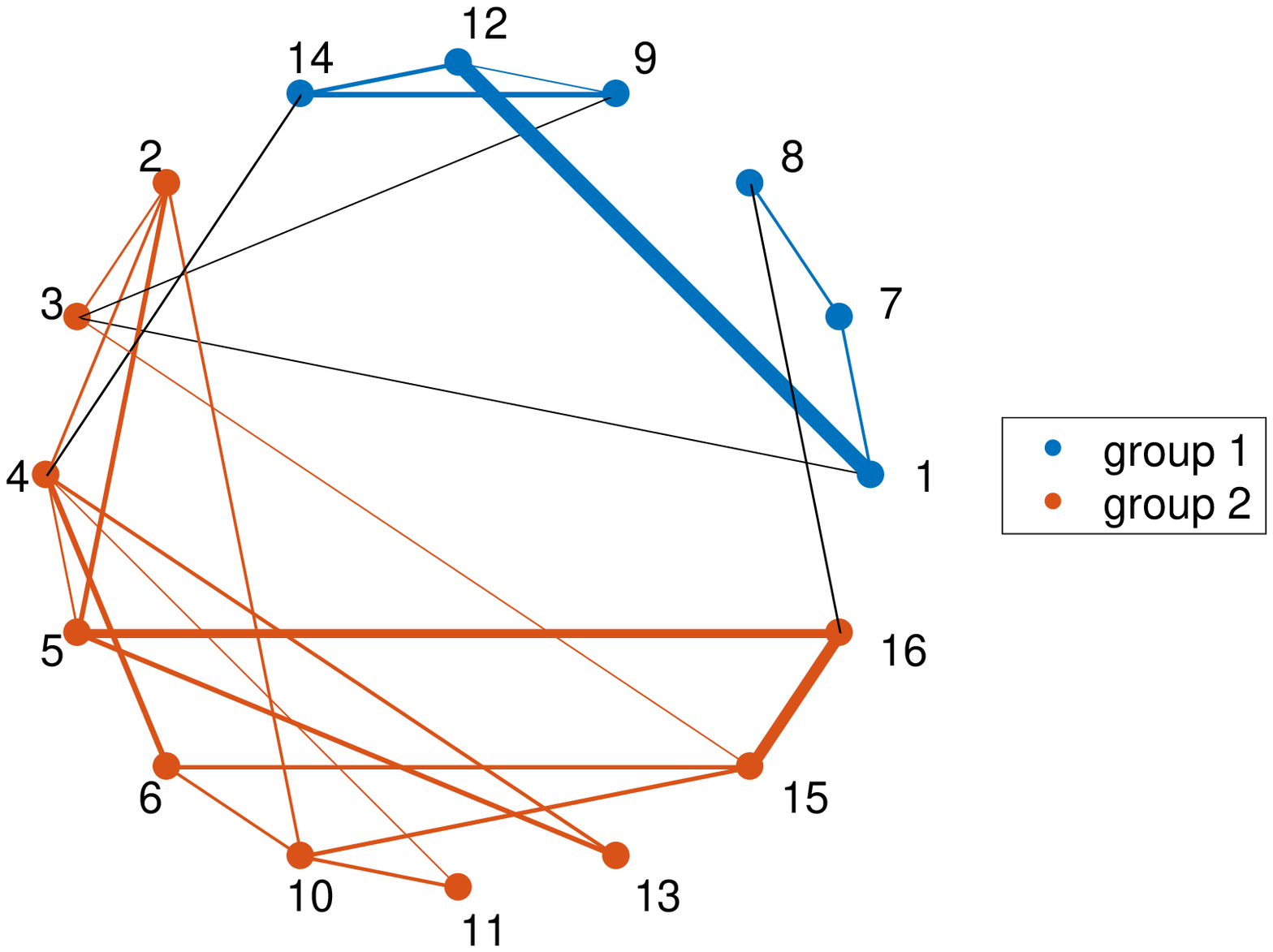}\label{F:LV_low_sim}
}\hspace{0.01in}
\subfigure[$\gamma_i=0.6$ (calculated)]{
\includegraphics[scale=0.25]{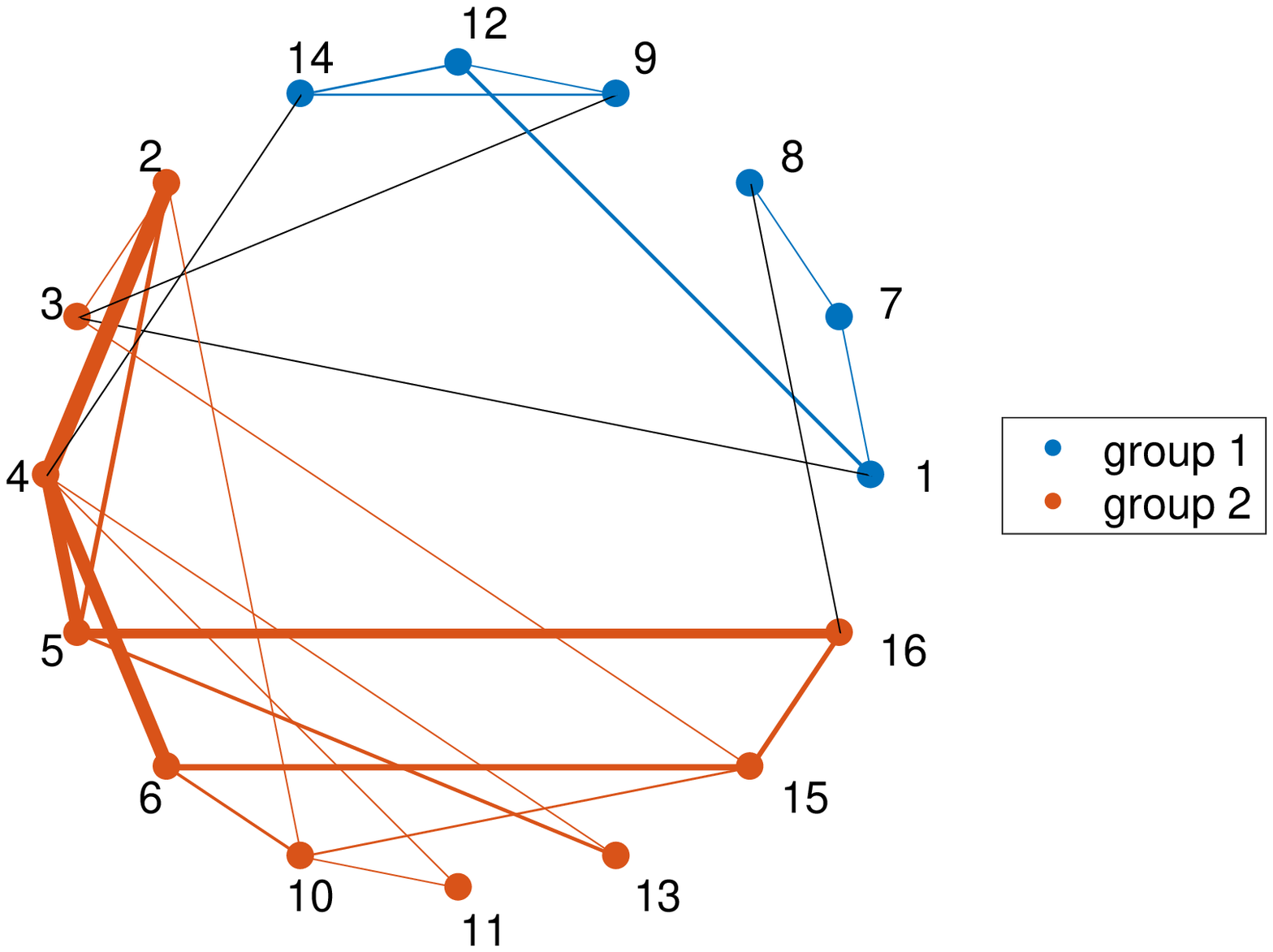}\label{F:LV_high_cal}
}
\hspace{0.01in}
\subfigure[$\gamma_i=0.6$ (simulated)]{
\includegraphics[scale=0.25]{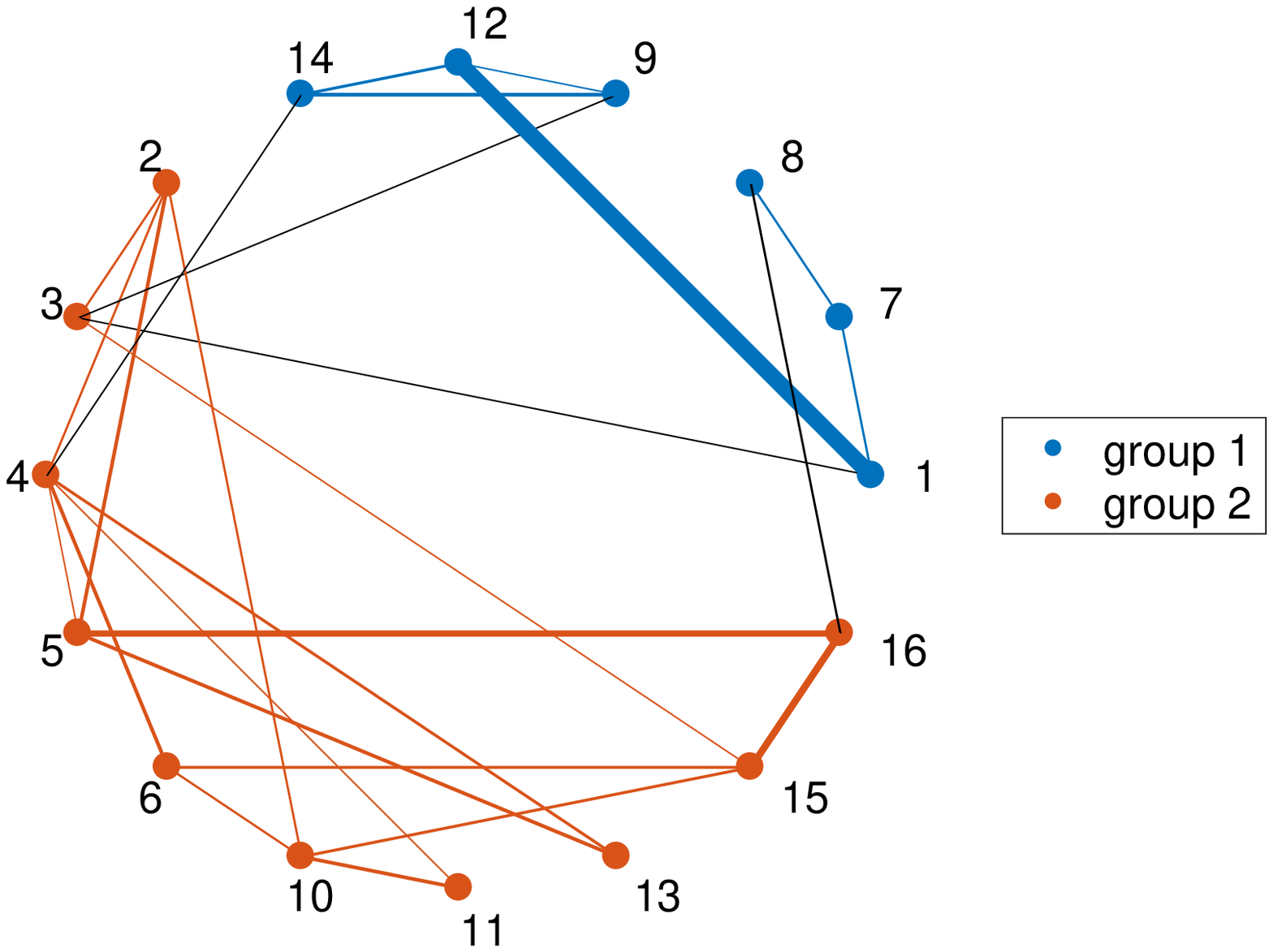}\label{F:LV_high_sim}
}
\caption[Optional caption for list of figures]{Calculated and simulated energy-flow patterns and decomposition for varying domains of interest in \eqref{E:LV}.}
\label{F:LV_low_high}
\end{figure*}

Fig.\,\ref{F:LV_low_high} shows calculated (based on the entries of the weighted adjacency matrix $W=[w_{ij}]$) and simulated (by randomly selecting an initial condition and evaluating the maximum value of $\psi_{ij}(0,\infty)$ over a number of scenarios) energy-flow patterns in the network for different domains of interest, and the associated decomposition into two sub-networks. When the domain of interest is defined to be a close neighborhood of the equilibrium point (by choosing $\gamma_i=0.01\,\forall i\,$), the simulated and calculated energy patterns look similar resulting in the same decomposition. When we choose $\gamma_i=0.6\,\forall i$\,, the calculated and simulated energy-flow patterns differ, although the decomposition remains the same. Note that, even when the energy-flow patterns differ in the calculation and simulation, the energy-flow between the sub-networks turns out to be small in both  cases. 

\begin{figure*}[thpb]
\centering
\captionsetup{justification=centering}
\subfigure[$V_i(x_i(0))$ high for group 1 nodes]{
\includegraphics[scale=0.29]{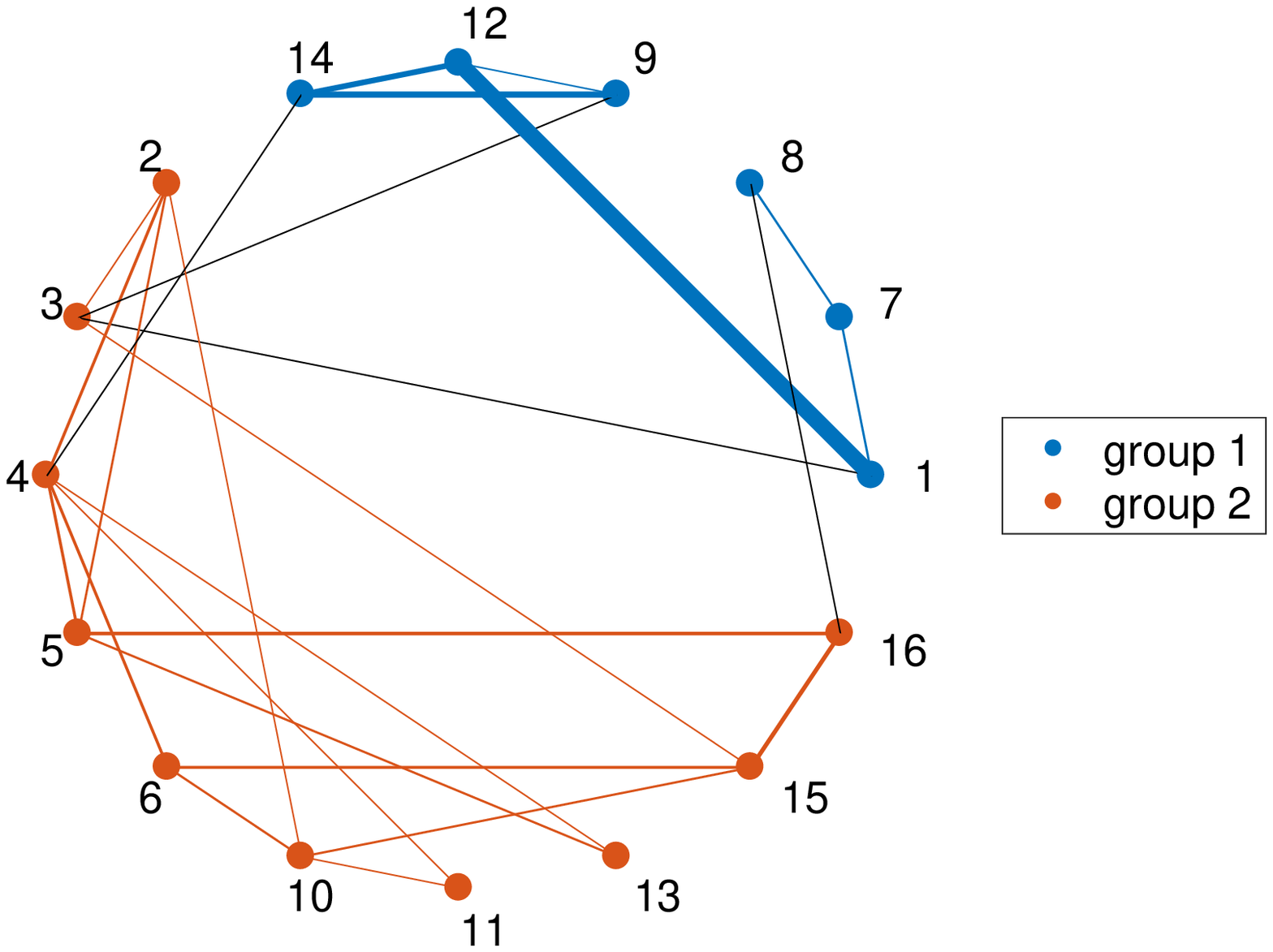}\label{F:LV_gr1}
}
\hspace{0.01in}
\subfigure[$V_i(x_i(0))$ high for group 2 nodes]{
\includegraphics[scale=0.29]{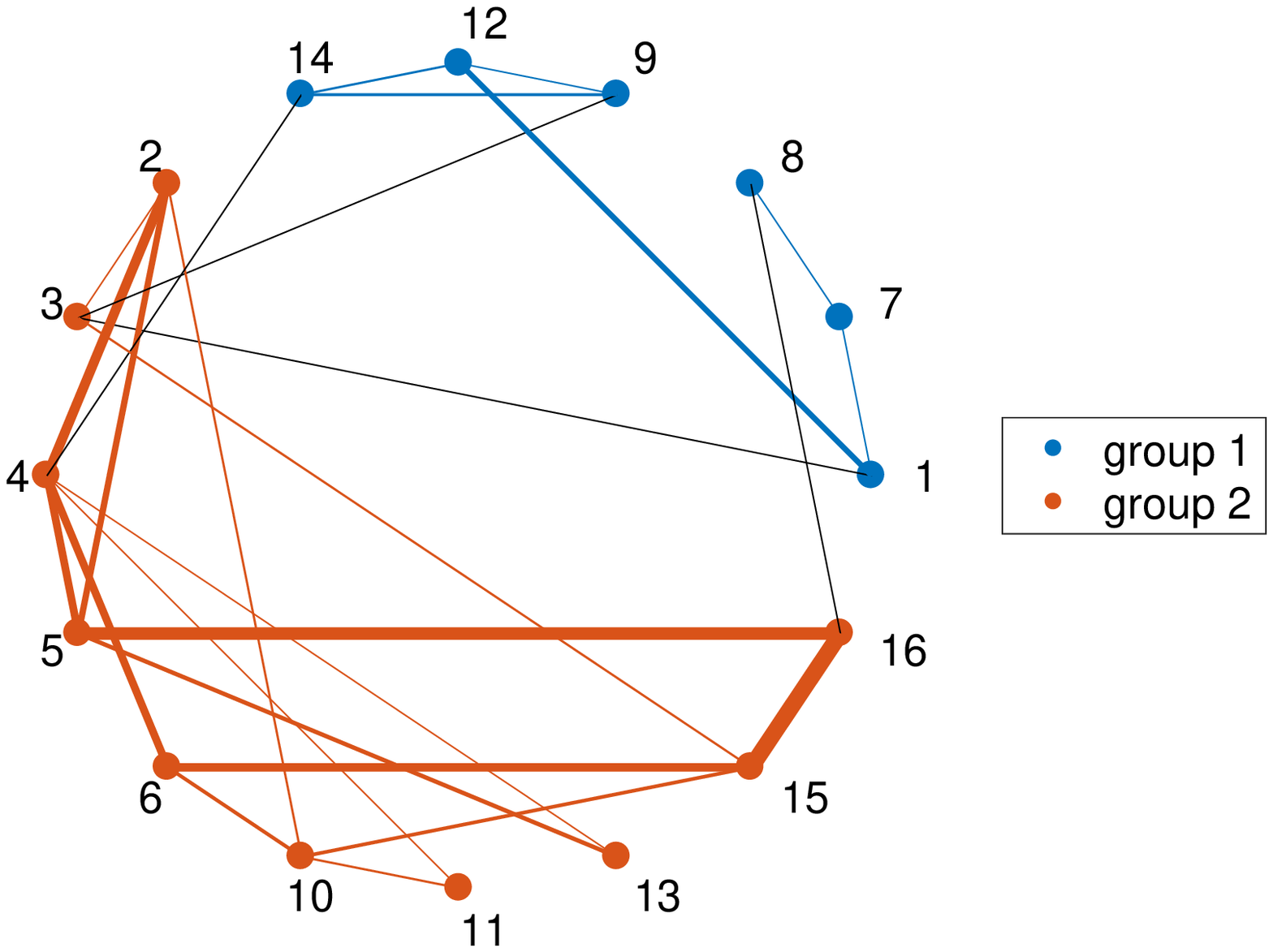}\label{F:LV_gr2}
}
\hspace{0.01in}
\subfigure[$V_i(x_i(0))$ high for $\!\lbrace 1,3,4,8,9,14,16\rbrace\!$]{
\includegraphics[scale=0.29]{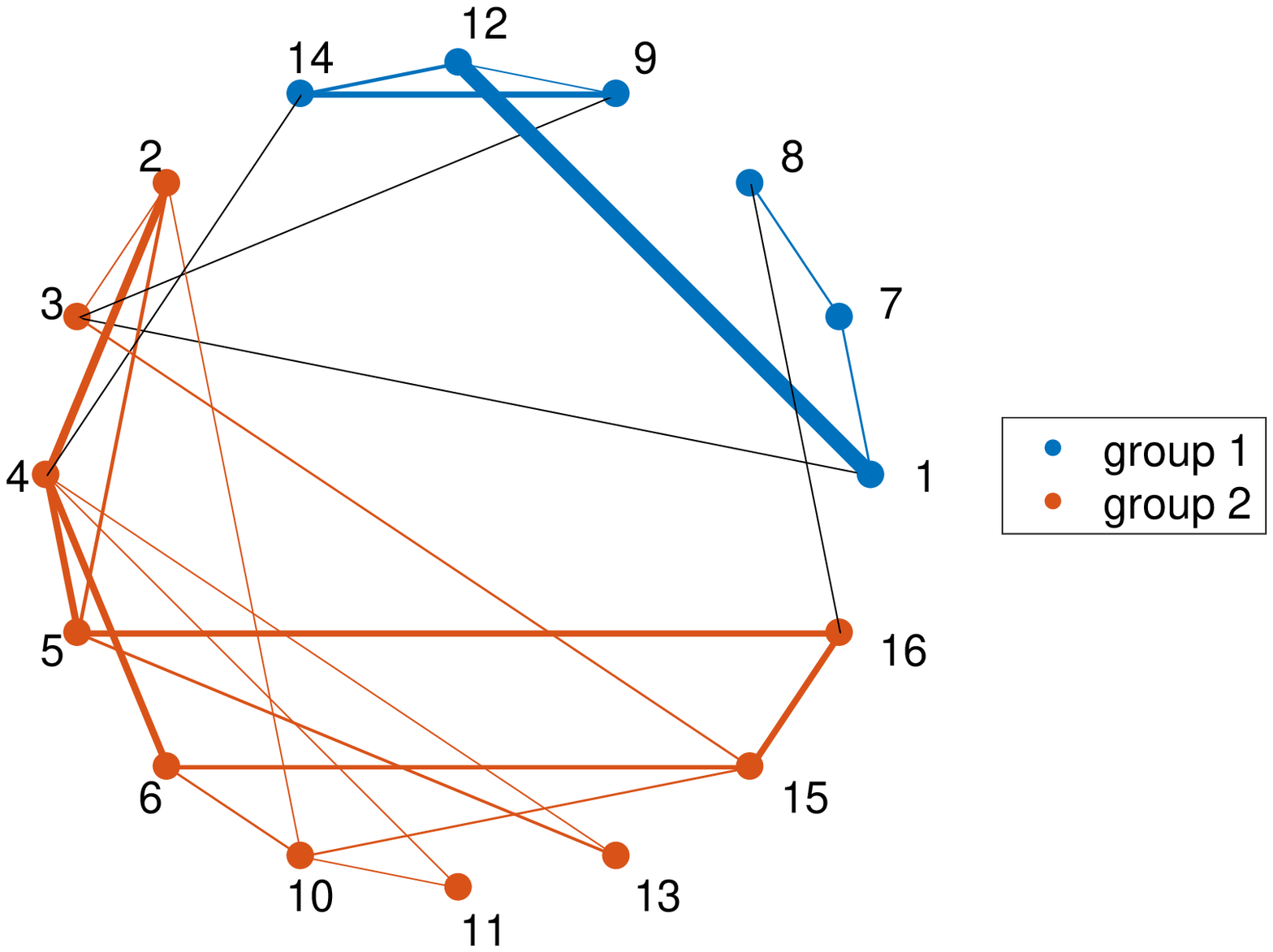}\label{F:LV_cross}
}
\caption[Optional caption for list of figures]{Calculated energy-flow patterns and decomposition of \eqref{E:LV} for varying initial conditions: high values of initial level-sets ($V_i(x_i(0))=0.5$) at selected nodes and low values ($V_i(x_i(0))=0.1$) at others.}
\label{F:LV_gr_cross}
\end{figure*}

Fig.\,\ref{F:LV_gr_cross} shows the energy-flow patterns when the initial condition of the network, determined by the initial energy level $V_i(x_i(0))$ of the nodes, is varied. Specifically, we consider three different scenarios - 1) when the initial level-sets of the vector LFs of the nodes in group-$1$ are set at high value of $0.5$ and the nodes in group-$1$ are assigned a low value of $0.1$ (Fig.\,\ref{F:LV_gr1}), 2) when group-$2$ nodes are assigned high initial level-sets and group-$1$ nodes low initial level-sets (Fig.\,\ref{F:LV_gr2}), and 3) when the level-sets of the nodes at the interconnection of the two sub-networks are assigned high initial level-sets and others low initial level-sets (Fig.\,\ref{F:LV_cross})\,. We can see that in all three scenarios, the energy-flows between two sub-networks remain small, while the energy-flow patterns within each group varies. This shows that such a decomposition indeed favors a distributed analysis and control design, by dividing the network into weakly interacting subsystems.

%---------------------------------------------------------------------------
\subsection{Network of Van der Pol Systems}
%---------------------------------------------------------------------------

Next we consider another example of a nonlinear network, composed of 9 Van der Pol systems from \cite{Kundu:2017Multiple}, described by
\begin{subequations}\label{E:VP}
\begin{align}
&\dot{x}_i=f_i(x_i)+\sum_{j\in\mathcal{N}_i\backslash\lbrace i\rbrace}g_{ij}(x_i,x_j)\\
&f_i(x_i)\!=\! \begin{bmatrix}
x_{i,2}\\
\mu_i\,x_{i,2}(c_i^{(1)}\!\!-\!c_i^{(2)}x_{i,1}\!-\!x_{i,1}^2) \!-\! c_i^{(3)}x_{i,1}\end{bmatrix} \\
&				g_{ij}(x_i,x_j)\!=\! \left[0\,,~\beta_{ij}^{(1)}\,x_{j,2} + \beta_{ij}^{(2)}\,x_{j,2}\,x_{i,1}\right]^T\,.
\end{align}\end{subequations}
where, $c_i^{(1)}\!=\!1\!-\!\left(0.5\,c_i^{(2)}\right)^2$, $c_i^{(3)}\!=\!1\!-\!{\sum}_{j\in\mathcal{N}_i\lbrace i\rbrace}({0.5\,\beta_{ij}^{(2)}c_i^{(2)}}\!-\!\beta_{ij}^{(1)})$, $\mu_i\,,\,\beta_{ij}^{(1)}$ and $\beta_{ij}^{(2)}$ are chosen randomly and $c_i^{(2)}$ are related to the equilibrium point before shifting. Vector LFs for the nine nodes are computed using sum-of-squares methods. Fig.\,\ref{F:VP_gr_cross} shows the energy-flow patterns and decomposition for varying operating conditions. Specifically, we choose a domain of interest by selecting $\gamma_i=0.6\,\forall i$\,, and plot the worst-case energy flow pattern in Fig.\,\ref{F:VP_high}, along with the decomposition. Keeping the same decomposition, we then monitor the change in energy-flow patterns as we vary the initial conditions across the network. In specific, we consider three scenarios: 1) choose high initial level-sets (equal to 0.6) for the group 1 nodes and low level-sets (equal to 0.1) for others (Fig.\,\ref{F:VP_gr1}), 2) high initial level-sets (=0.6) for group 2 nodes and low (=0.1) for others (Fig.\,\ref{F:VP_gr2}), and 3) high level-sets (=0.6) only for the nodes 1,\,5,\,7 and 8 that connect the two sub-networks. We observe that most of the energy-flow is contained within the sub-network, with minimal flow between the two sub-networks, implying that the decomposition yields a weakly interacting network.

\begin{figure*}[thpb]
\centering
\captionsetup{justification=centering}
\subfigure[$\gamma_i=0.6$ for each node]{
\includegraphics[scale=0.30]{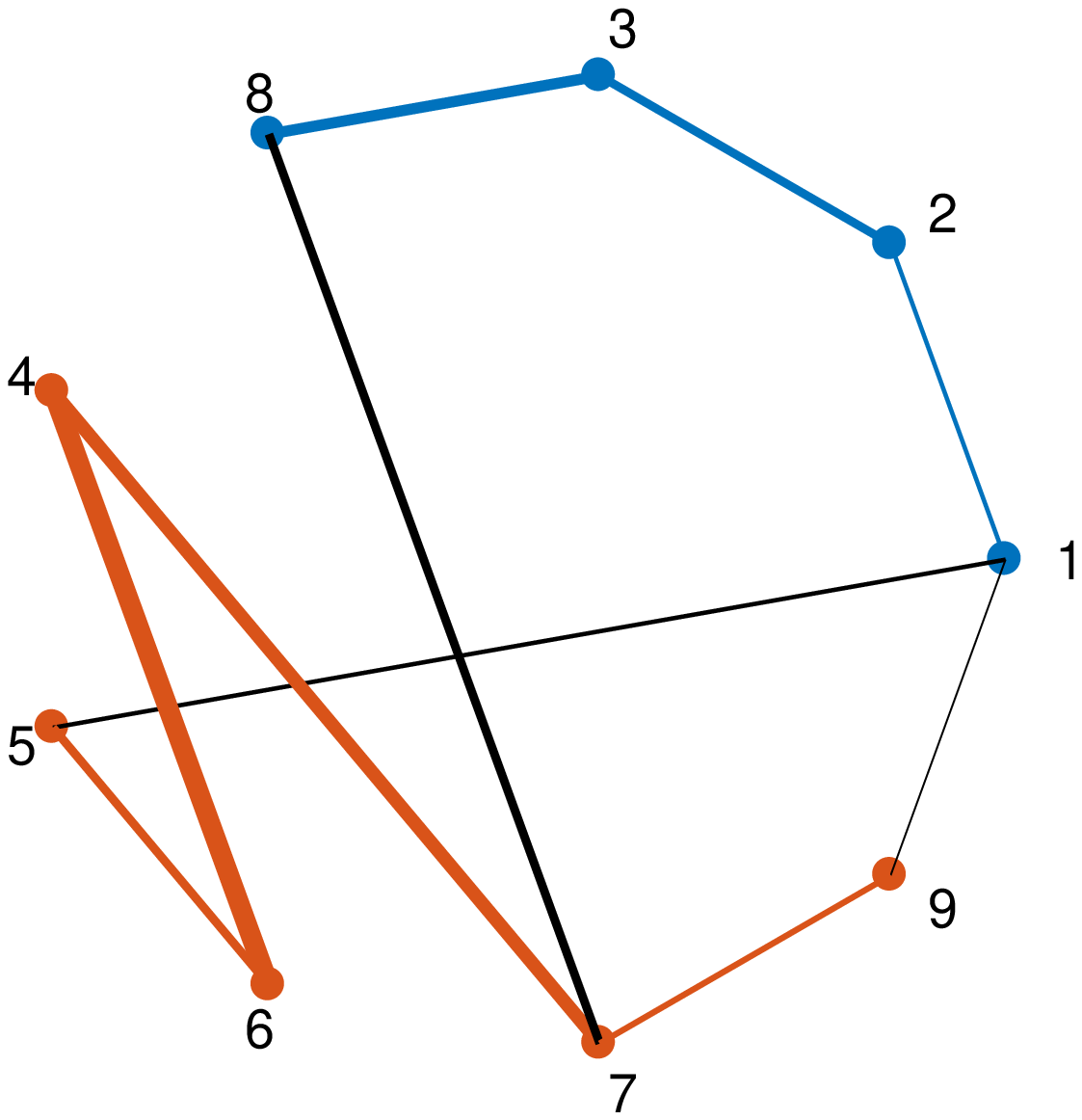}\label{F:VP_high}
}
\hspace{0.01in}
\subfigure[$V_i(x_i(0))$ high for group 1]{
\includegraphics[scale=0.30]{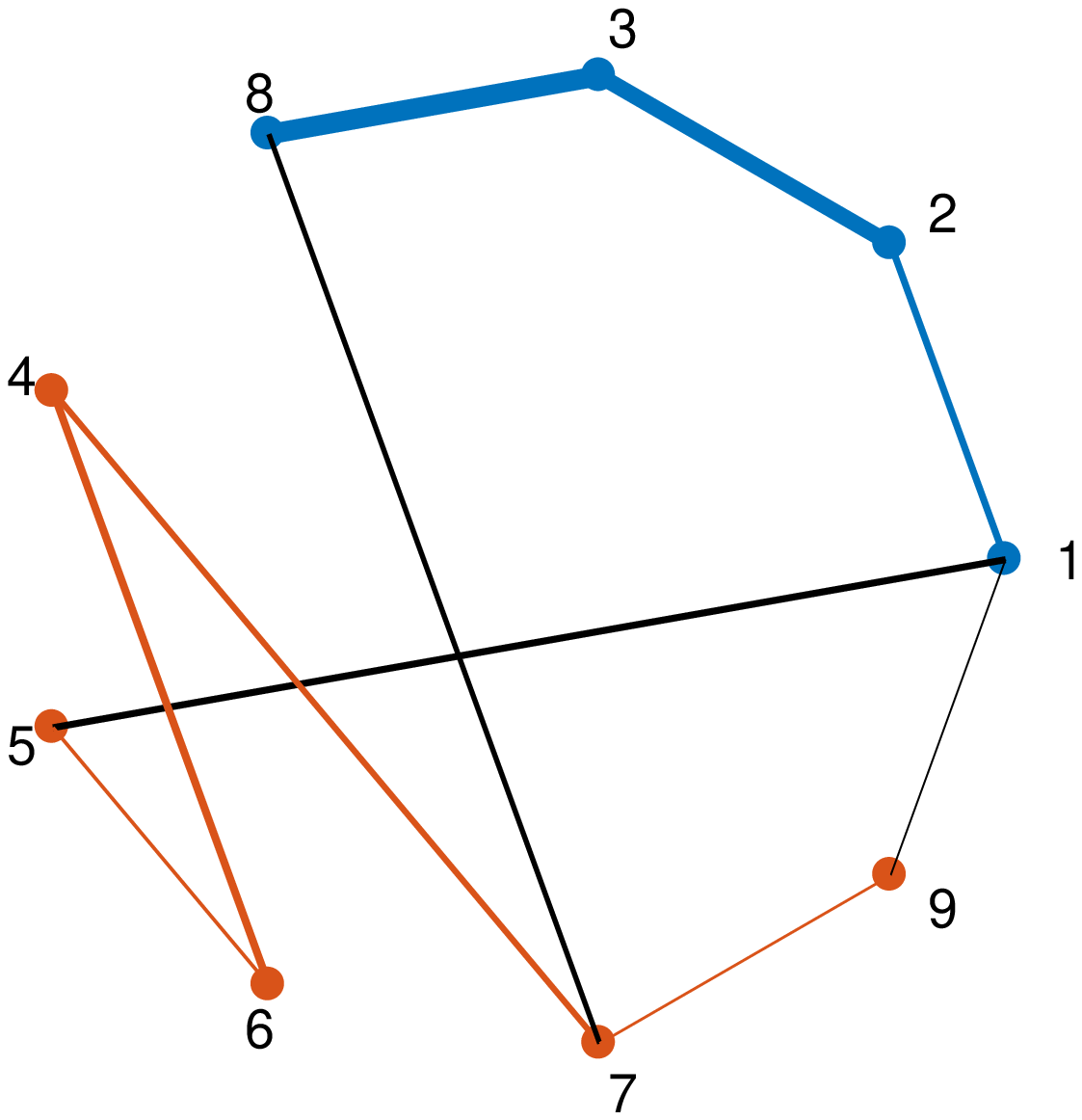}\label{F:VP_gr1}
}\hspace{0.01in}
\subfigure[$V_i(x_i(0))$ high for group 2]{
\includegraphics[scale=0.30]{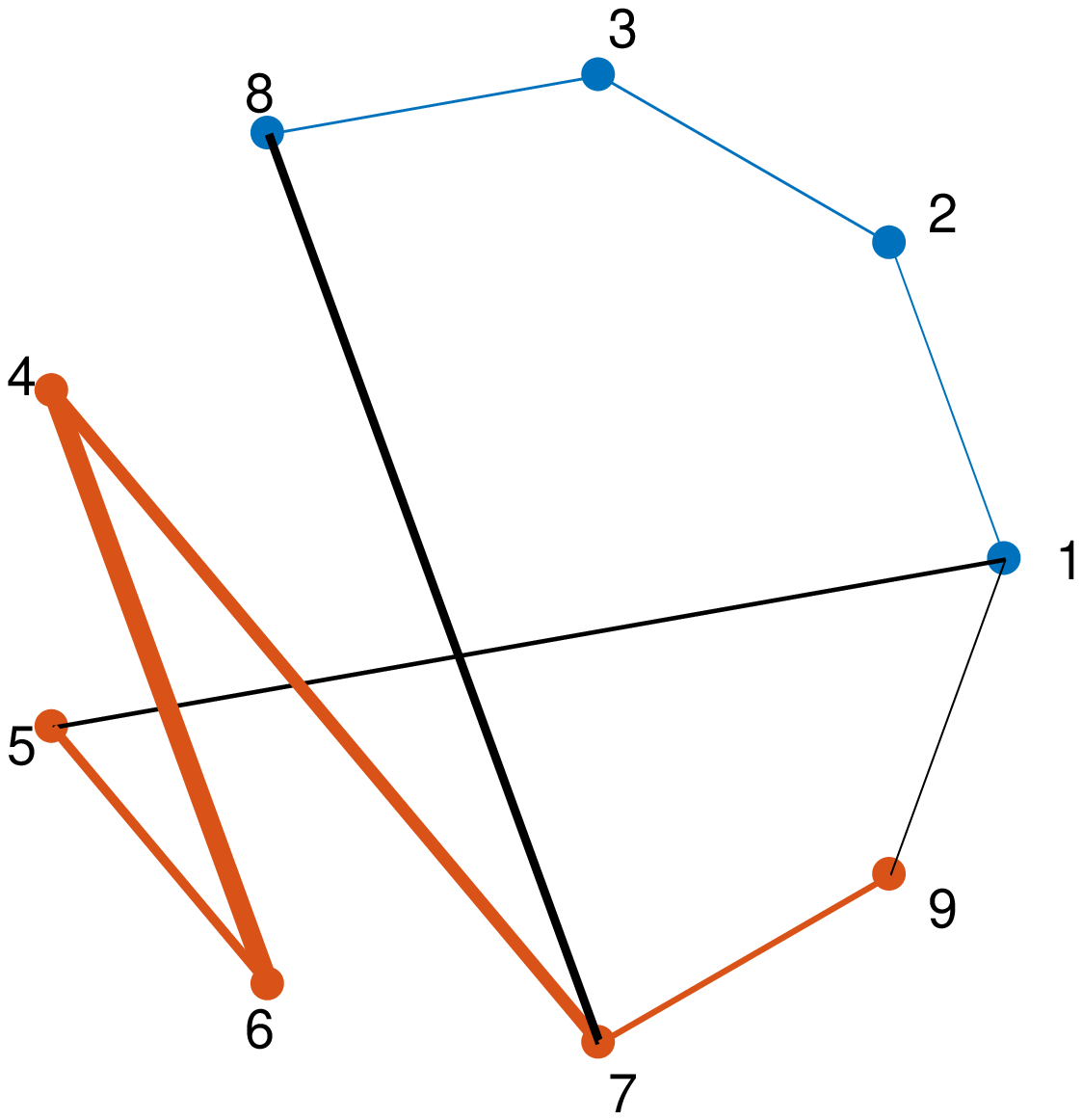}\label{F:VP_gr2}
}
\hspace{0.01in}
\subfigure[$V_i(x_i(0))$ high for $\!\!\lbrace 1,5,7,8\rbrace\!$]{
\includegraphics[scale=0.30]{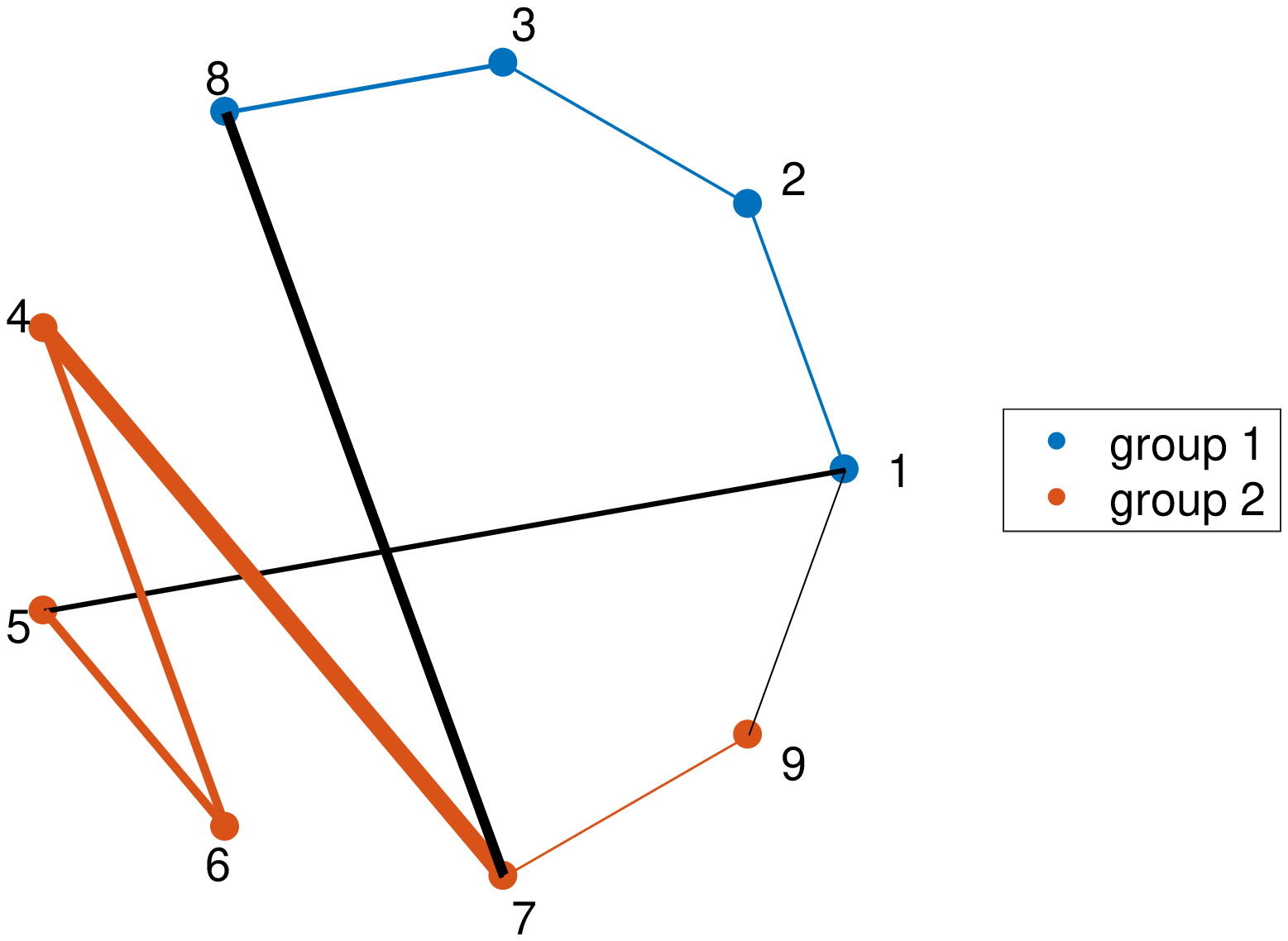}\label{F:VP_cross}
}
\caption[Optional caption for list of figures]{Calculated energy-flow patterns and decomposition for \eqref{E:VP}, under varying initial conditions (defined by $V_i(x_i(0))\,$).}
\label{F:VP_gr_cross}
\end{figure*}

\section{Conclusion}\label{S:concl}
In this article we considered the problem of decomposing a nonlinear network into weakly interacting subsystems, to facilitate distributed analysis and control design. Using a vector Lyapunov functions based approach we showed that the evolution of the energy levels of the nodes of the system can be modeled via a linear comparison system. Further, introducing a notion of power and energy flows between two nodes in a dynamical network, we proposed a method to compute the dynamic edge weights in the network. Finally, using spectral clustering techniques on the weighted adjacency matrix the nonlinear dynamical network is decomposed into weakly interacting subsystems. Sum-of-squares programming tools are used to demonstrate the working on the algorithm on two examples of nonlinear networks. Future work will investigate the applicability of this method to larger systems (interested readers are referred to the recent work in \cite{Ahmadi:2017}), and to real-world problems such as coherency detection in power systems.


\section*{Acknowledgment}

This work was supported by the United States Department
of Energy (contract DE-AC02-76RL01830) under the Control of Complex Systems Initiative at Pacific Northwest National Laboratory.



\bibliographystyle{IEEEtran}
\bibliography{IEEEabrv,references,RefKundu}
%\begin{thebibliography}{10}
%\end{thebibliography}

\end{document}